\documentclass[11pt,oneside,reqno]{article}
\usepackage[a4paper,width=15cm,height=24cm]{geometry}

\usepackage{amsmath,amsthm,amssymb,mathtools}
\usepackage[authoryear]{natbib}
\usepackage{bbm}
\usepackage{mathrsfs}
\usepackage{textcase}
\usepackage{paralist}
\usepackage{float}
\usepackage[font=small,labelfont=bf]{caption}
\usepackage{subfigure}
\usepackage{type1cm}
\usepackage[hyperindex,breaklinks]{hyperref}

\usepackage{tikz,pgf}

\numberwithin{equation}{section}
\allowdisplaybreaks[4]

\theoremstyle{plain}
\newtheorem{theorem}{Theorem}[section]

\newtheorem{proposition}[theorem]{Proposition}
\newtheorem{lemma}[theorem]{Lemma}
\newtheorem{corollary}[theorem]{Corollary}

\theoremstyle{definition}

\newtheorem{remark}[theorem]{Remark}

\makeatletter

\renewcommand{\cite}{\citet*}

\def\^#1{\ifmmode {\mathaccent"705E #1} \else {\accent94 #1} \fi}
\def\~#1{\ifmmode {\mathaccent"707E #1} \else {\accent"7E #1} \fi}

\def\*#1{#1^\ast}
\edef\-#1{\noexpand\ifmmode {\noexpand\bar{#1}} \noexpand\else
\-#1\noexpand\fi}\def\>#1{\vec{#1}}
\def\.#1{\dot{#1}}

\def\atop{\@@atop}
\def\%#1{\mathcal{#1}}

\renewcommand{\leq}{\leqslant}
\renewcommand{\geq}{\geqslant}
\renewcommand{\phi}{\varphi}
\newcommand{\eps}{\varepsilon}

\newcommand{\eq}{\eqref}

\newcommand{\dk}{\mathop{d_{\mathrm{Kol}}}}

\newcommand{\bigo}{\mathrm{O}}

\def\tsfrac#1#2{{\textstyle\frac{#1}{#2}}}
\newcommand{\toinf}{\to\infty}

\newcommand{\Bi}{\mathop{\mathrm{Bi}}}

\newcommand{\GG}{\mathop{\mathrm{GGa}}}
\newcommand{\Ga}{{\mathrm{Gamma}}}

\newcommand{\IE}{\mathbbm{E}}
\newcommand{\IP}{\mathbbm{P}}

\newcommand{\law}{\mathscr{L}}
\newcommand{\eqlaw}{\stackrel{\mathscr{D}}{=}}

\newcommand{\IZ}{\mathbbm{Z}}
\newcommand{\IR}{\mathbbm{R}}

\newcommand{\I}{{\mathrm{I}}}

\def\be#1{\begin{equation*}#1\end{equation*}}
\def\ben#1{\begin{equation}#1\end{equation}}
\def\bes#1{\begin{equation*}\begin{split}#1\end{split}\end{equation*}}
\def\besn#1{\begin{equation}\begin{split}#1\end{split}\end{equation}}

\def\ba#1{\begin{align*}#1\end{align*}}

\def\klr#1{(#1)}
\def\bklr#1{\bigl(#1\bigr)}
\def\bbklr#1{\Bigl(#1\Bigr)}

\def\kle#1{[#1]}
\def\bkle#1{\bigl[#1\bigr]}
\def\bbkle#1{\Bigl[#1\Bigr]}
\def\bbbkle#1{\biggl[#1\biggr]}

\def\klg#1{\{#1\}}
\def\bklg#1{\bigl\{#1\bigr\}}
\def\bbklg#1{\Bigl\{#1\Bigr\}}
\def\bbbklg#1{\biggl\{#1\biggr\}}

\def\norm#1{\Vert#1\Vert}

\def\abs#1{\vert#1\vert}
\def\babs#1{\bigl\vert#1\bigr\vert}

\def\bbbabs#1{\biggl\vert#1\biggr\vert}

\def\bbbmid{\biggm\vert}

\def\floor#1{{\lfloor#1\rfloor}}

\def\ceil#1{{\lceil#1\rceil}}

\def\ellp{p}

\newcount\minute
\newcount\hour
\newcount\hourMins
\def\now{%
\minute=\time%
\hour=\time \divide \hour by 60%
\hourMins=\hour \multiply\hourMins by 60%
\advance\minute by -\hourMins%
\zeroPadTwo{\the\hour}:\zeroPadTwo{\the\minute}%
}
\def\zeroPadTwo#1{\ifnum #1<10 0\fi#1}

\renewcommand\section{\@startsection {section}{1}{\z@}%
{-3.5ex \@plus -1ex \@minus -.2ex}%
{1.3ex \@plus.2ex}%
{\center\small\sc\MakeTextUppercase}}

\def\subsection#1{\@startsection {subsection}{2}{0pt}%
{-3.5ex \@plus -1ex \@minus -.2ex}%
{1ex \@plus.2ex}%
{\bf\mathversion{bold}}{#1}}

\def\subsubsection#1{\@startsection{subsubsection}{3}{0pt}%
{\medskipamount}%
{-10pt}%
{\normalsize\itshape}{\kern-2.2ex. #1.}}

\def\blfootnote{\xdef\@thefnmark{}\@footnotetext}

\makeatother

\def\cF{\mathcal{F}}

\def\B{\mathrm{B}}
\def\ed{\eqlaw}

\def\B{{\mathrm{Beta}}}

\def\bPclas#1#2#3{{\mathcal{P}\bklr{{\textstyle{#1\atop #2}};#3}}}
\def\bbPclas#1#2#3{{\mathcal{P}\bbklr{{\textstyle{#1\atop #2\strut}};#3}}}

\def\bPimm#1#2#3#4{{\mathcal{P}_{\mathrm{Im}}^#1\bklr{{\textstyle{#2\atop #3}};#4}}}
\def\bbPimm#1#2#3#4{{\mathcal{P}_{\mathrm{Im}}^#1\bbklr{{\textstyle{#2\atop #3\strut}};#4}}}

\usepackage{hyperref}
\begin{document}

\title{\sc\bf\large\MakeUppercase{
Joint degree distributions of preferential attachment random graphs
}}
\author{\sc Erol Pek\"oz\thanks{Boston University, Questrom School of Business, 595 Commonwealth Avenue, Room 607
Boston, MA 02215, USA; \texttt{pekoz@bu.edu}}, Adrian R\"ollin\thanks{National University of Singapore, Department of Statistics and Applied Probability,
6 Science Drive 2,
Singapore, 117546, Singapore; \texttt{adrian.roellin@nus.edu.sg}}
and Nathan Ross\thanks{University of Melbourne, School of Mathematics and Statistics, Richard Berry Building,
University of Melbourne, VIC, 3010, Australia; \texttt{nathan.ross@unimelb.edu.au}}
\date{
}
}
\maketitle

\begin{abstract}
We study the joint degree counts in linear preferential attachment random graphs and find a simple representation for the limit distribution in infinite sequence space. We show weak convergence with respect to the~$p$-norm topology for appropriate~$p$ and also provide optimal rates of convergence of the finite dimensional distributions. The results hold for models with any general initial seed graph and any fixed number of initial outgoing edges per vertex; we generate non-tree graphs using both a lumping and a sequential rule. Convergence of the order statistics and optimal rates of convergence to the maximum of the degrees is also established.
\end{abstract}


\section{Introduction}

Preferential attachment random graph models
have become extremely popular in the fifteen years since they were studied by \cite{Barabasi1999}.
In the basic models, nodes are sequentially added to the network over
time and connected randomly to existing nodes such that connections
to higher degree nodes are more likely.  The literature around these models
has become too vast to survey, but \cite{Hofstad2013}, \cite{Newman2003} and \cite{Newman2006}
provide good overviews.

The most popular models are those similar to \cite{Barabasi1999}, in which
nodes are added sequentially and attach to exactly one randomly chosen existing node,
and the chance a new node connects to an existing node is proportional to its degree; 
see \cite{Collevecchio2013}, \cite{Krapivsky2000}, and \cite{Rudas2007} for
results on more general attachment rules. The model is typically generalized
to allow for vertices to have~$\ell\geq1$ initial edges by starting with the previous model and then, for~$k=0,1,2,\ldots$,
lumping vertices~$k\ell+1, k\ell+2, \ldots, (k+1)\ell$ into
a single vertex (possibly causing loops). The most studied feature of these objects is
the distribution of the degrees of the nodes; that is, the proportion of nodes that have degree~$k$ as
the graph grows large. The basic content of \cite{Barabasi1999} and the rigorous formulation
of \cite{Bollobas2001} is that, as~$k\to\infty$, this distribution
roughly decays proportional to~$k^{-\gamma}$ for some~$\gamma>0$; this is the so-called power law behavior.

In this article we study the joint degree distribution 
for a linear preferential attachment model where each entering node initially attaches to exactly
$\ell \geq 1$ nodes. In the case~$\ell\geq 2$, we study two mechanisms for attaching edges, a sequential
update rule and the lumping rule mentioned above (and note our results are strikingly different for the
two rules).
We show weak convergence with respect to~$p$-norm topology 
of the scaled degrees for the process started
from any initial ``seed" graph to a limiting distribution that has a simple representation
and provide an optimal rate of convergence for the finite dimensional distributions.

To state our results in greater detail, we first precisely define the random rules governing the evolution of the models we study.
We distinguish between two cases: one that allows for loops and
the other that does not; in fact, the two can be related (see Lemma~\ref{lemlink} below), but we present the results separately for
the sake of clarity.

Our results are stated in terms of weights of vertices, but our weights
can be thought of as in-degree plus one since in the models, vertices
are born with weight one, and
each time a vertex receives a new edge from another vertex, its weight increases by one.
Fix~$\ell\geq 1$ and let~$d_k(n)$ denote the weight of vertex~$k$ in the graph~$G(n)$.
Assume that the seed graph~$G(0)$ has~$s$
vertices with labels~$1,\ldots,s$ and with initial weights~$d_1,\dots,d_s$ (note that~$d_i=d_i(0)$).
We construct the graph~$G(n)$ with~$n$ vertices from~$G(n-1)$ having~$n-1$~edges in two possible ways.

\paragraph{Model~N$_{\boldsymbol{\ell}}$.}
Given the graph~$G(n-1)$ having~$s+n-1$ vertices,
$G(n)$ is formed by adding a vertex labeled~$s+n$ and sequentially attaching~$\ell$
edges between it and the
vertices of~$G(n-1)$ according to the following rules. The first edge attaches
to vertex~$k$ with
probability
\be{
\frac{d_k(n-1)}{\sum_{i=1}^{s+n-1}d_i(n-1)}, \qquad 1\leq k\leq n-1;
}
denote by~$K_1$ the vertex which received that first edge. The weight of~$K_1$ is updated immediately, so that
the second edge attaches to vertex~$k$ with probability
\be{
\frac{d_{k}(n-1)+\I[k=K_1]}{1+\sum_{i=1}^{s+n-1}d_i(n-1)}, \qquad 1\leq k\leq n-1.
}
The procedure continues this way,
edges attach with probability proportional to weights at that moment, and additional
received edges add one to the weight of a vertex, until vertex~$n$ has~$\ell$ outgoing edges.
Lastly, we set~$d_{s+n}(n)=1$, and let~$G(n)$ be the resulting graph. Note that multiple
edges between vertices are possible.

\paragraph{Model~L$_{\boldsymbol{\ell}}$.}
Given the graph~$G(n-1)$ having~$s+n-1$ vertices,
$G(n)$ is formed by adding a vertex labeled~$s+n$ having weight~$d_{s+n}(n-1)=1$
and attaching~$\ell$ edges between it and the vertices labelled ~$\{1, \ldots, s+n\}$
(so loops are possible). The first edge attaches to vertex~$k$ with probability

\be{
\frac{d_k(n-1)}{\sum_{i=1}^{s+n}d_i(n-1)}; \qquad 1\leq k\leq n;
}
denote by~$K_1$ the vertex which received that first edge.
As in model N$_\ell$, the weight of~$K_1$ is updated immediately, and the procedure continues in this way until~$\ell$ edges are added
and we call~$G(n)$ the resulting graph.

\bigskip

Note that if~$\ell=1$, then the weights can be interpreted as total degree since each additional vertex
has ``out-degree" one and in this case the~$N_1$ model is just the usual Barab{\'a}si-Albert preferential attachment tree.

In the next section we state our finite dimensional results; process level 
statements are in Section~\ref{sec:introproc}.

\subsection{Finite dimensional degree distributions}

Key to both understanding and obtaining our results
is to focus on the joint cumulative degree counts rather than the joint degree counts. 
Write
$X\sim\GG(a,b)$ with~$a>0$ and~$b>0$ for a random variable~$X$ having the
generalized gamma distribution with density proportional to~$x^{a-1}e^{-x^b}$ on~$x>0$,
and~$Y\sim\B(a,b)$ for~$a>0$ and~$b>0$ if
$Y$ has density proportional to~$x^{a-1}(1-x)^{b-1}$ on~$0<x<1$.

\begin{theorem}\label{THM2}
Fix a seed graph~$G(0)$ having~$s$ vertices and  weight sequence~$d_1,\dots,d_s$, and let~$m_k=\sum_{i=1}^k d_i$ for~$1\leq k\leq s$.
Assume that either
\begin{enumerate}[$(i)$]
\item~$G(n)$ follows Model~N$_\ell$, in which case let~$a_k=m_s+(\ell+1)(k-s)+\ell$ for~$k\geq s$; or
\item~$G(n)$ follows Model~L$_\ell$, in which case let~$a_k=m_s+(\ell+1)(k-s)$ for~$k\geq s$.
\end{enumerate}
Fix~$r \geq s$ and let~$B_1,\dots,B_{r-1}$ and~$Z_r$ be independent random variables with distributions
\be{
	B_k\sim  \begin{cases}
	\B(m_{k},d_{k+1}) & \text{if~$1\leq k < s$,} \\
	\B(a_k,1) & \text{if~$s\leq k\leq r$,}
	\end{cases}
}
and~$Z_r \sim \GG(a_r,\ell+1)$. Define the products
\be{
	Z_k = B_{k}\cdots B_{r-1} Z_r, \quad 1\leq k < r, 
}
and let~$Z=(Z_1,\dots,Z_r)$ and~$Y = (Z_1,Z_2-Z_1,\dots,Z_r-Z_{r-1})$.
Denote the scaled weight sequence of the first~$r$ vertices of~$G(n)$ by
\be{
	D(n)=\bklr{D_1(n),\ldots, D_r(n)}=\frac{1}{(\ell+1)n^{\ell/(\ell+1)}}\bklr{d_1(n),\ldots, d_r(n)}.
}
Then there is a positive constant~$C = C(r,\ell,m_s)$ such that
\be{
	\sup_{K}\babs{\IP[D(n)\in K]-\IP[Y\in K]} \leq \frac{C}{n^{\ell/(\ell+1)}}
}
for all~$n\geq 1$, where the supremum ranges over all convex subsets~$K\subset\IR^r$.
\end{theorem}

\begin{remark}
The error rate~$n^{-\ell/(\ell+1)}$ is best-possible since the rate of convergence of a scaled integer valued random variable
to a limiting distribution with nice density is bounded from below by the scaling (nice means uniformly bounded away from zero on some interval; see
\cite[Lemma~4.1]{Pekoz2013}). Also, a bound on the constant
$C(r,\ell,m_s)$ could in principle be made explicit with our methods, but with much added technicality. Such a constant would increase in each of its arguments.
We emphasize that obtaining optimal rates in multidimensional distributional limit theorems
for the metric we are using is in general neither easy nor common.
\end{remark}

Since the sets~$\{(x_1,\ldots, x_r):\max\{ x_1,\ldots, x_r\} \leq t\}$ are convex in~$\IR^r$, we immediately obtain the following corollary to the theorem.

\begin{corollary}\label{cor1}
Let\/~$D$ and\/~$Y$ be as in Theorem~\ref{THM2} under either~$(i)$ or~$(ii)$. Then, there is a positive constant~$C=C(r,\ell,m_s)$, such that
\be{
\sup_{t\geq 0}\babs{\IP\bkle{\max\nolimits_{1\leq k\leq r} D_k(n) \leq t}- \IP\bkle{\max\nolimits_{1\leq k\leq r} Y_k\leq t}}\leq \frac{C}{n^{\ell/(\ell+1)}}
}
for all~$n\geq 1$.
\end{corollary}

The representation of the limits appearing in Theorem~\ref{THM2}
is a sort of ``backward" construction started from~$Z_r$. Our next result is 
an alternative ``forward" construction that is useful
for describing the infinite limit~$\law(Z_1,Z_2,\ldots)$; cf.\ the forthcoming process level discussion.
Denote by~$\Ga(r,\lambda)$ a gamma distribution with shape parameter~$r>0$ and rate parameter~$\lambda>0$
having density proportional to~$x^{r-1} e^{-\lambda x}$,~$x>0$.

\begin{proposition}\label{thm1} 
Let~$B_1,\ldots, B_{s-1}$ and~$a_s$ be as in Theorem~\ref{THM2},
and let~$X_1,X_2,\ldots$ be independent with 
$X_1\sim \Ga(a_s/(\ell+1),1)$ and~$X_k\sim\Ga(1,1)$,~$k\geq 2$.
Then the random vector~$Z$ in Theorem~\ref{THM2}
has representation~$Z\eqlaw \tilde{Z}$, where
\be{
	\tilde{Z_k} =  \begin{cases}
	B_{k}\cdots B_{s-1} \, X_1^{1/(\ell+1)} & \text{if~$1\leq k < s$,} \\
	(X_1+\cdots+X_{k-s+1})^{1/(\ell+1)} & \text{if~$k\geq s$.}
	\end{cases}
}
\end{proposition}
\begin{proof}
We need to check that  the representation above has the
same distribution as that of Theorem~\ref{THM2}.
The two representations are identical for~$k=1,\ldots, s-1$.
For~$k\geq s$, that the joint distributions continue to agree
is an easy consequence of induction and the beta-gamma algebra which implies that for 
$k>s$,
\bes{
(X_1+\cdots+X_{k-s})^{1/(\ell+1)}&=\left(\frac{X_1+\cdots+X_{k-s}}{X_1+\cdots+X_{k-s+1}}\right)^{1/(\ell+1)}\!\!\cdot\left(X_1+\cdots+X_{k-s+1}\right)^{1/(\ell+1)}\\
&\ed V^{1/(\ell+1)} Z_k,
}
where~$V\sim\B(\tsfrac{a_s}{\ell+1}+k-s-1,1)$ is independent of~$Z_k$. A simple calculation
shows that~$V^{1/(\ell+1)}\sim \B(a_s+(\ell+1)(k-s-1),1)$ which is the same distribution as~$B_{k-1}$.
Continuing in this way yields the corollary.
\end{proof}

\begin{remark}
Proposition~\ref{thm1} 
leads to rather clean representations of the limits appearing in Theorem~\ref{THM2}
for particular choices of seed graph.
If
 the seed graph~$G(0)$ has one vertex with~$d_1=\ell+1$ 
 and~$G(n)$ is formed according to Model~L$_\ell$ (alternatively the seed graph has two vertices with~$d_1=\ell+1$ and~$d_2=1$,
 following Model~N$_\ell$), then Proposition~\ref{thm1} 
implies that~$0<Z_1<Z_2<\cdots$ are the points of an inhomogeneous Poisson point process
on the positive line with intensity
$(\ell+1)t^\ell dt$.
\end{remark}

\subsection{Process level convergence and order statistics}\label{sec:introproc}

Using Kolmogorov's extension theorem, for any choice of seed graph and generative mechanism,
the distribution of the vector~$Y$ in Theorem~\ref{THM2} can be uniquely extended to obtain an infinite vector~$\tilde Y=(Y_1,Y_2,\dots)$ which we can take to be of the form~$Y_i=Z_i-Z_{i-1}$
for~$0<Z_1<Z_2<\cdots$ given explicitly by the forward construction of Proposition~\ref{thm1}.
It is immediate from Theorem~\ref{THM2} that~$\tilde D(n)=(D_1(n),\dots,D_n(n),0,\dots)$,
viewed as an infinite sequence by appending zeros, converges weakly to~$\tilde Y$ with respect to the product topology. 
However, weak convergence with respect to this topology does not yield much at the process level (for example, 
restricting to bounded sequences, it does not imply convergence of the sequence of maximums to the maximum of the limit), and so we show weak convergence with respect to a topology that is strong enough to imply weak convergence of order statistics. 

\def\da{\downarrow}
For~$1\leq p<\infty$, let~$l_p$ be the usual sequence-space endowed with the norm~$\norm{\cdot}_p$. Moreover, denote by~$l_p^+\subset l_p$ the subspace of sequences having non-negative entries along with the topology inherited from~$l_p$. Clearly,~$\tilde D(n)\in l_p^+$ for all~$n\geq 1$. For~$x\in l_p^+$, denote by~$x^\da$ the sequence with the entries of~$x$ appearing in decreasing order; clearly~$x^\da\in l_p^+$.

We have the following result
which greatly generalizes the convergence of the maximum of the first~$r$ coordinates of~$D(n)$
given in Corollary~\ref{cor1} to the order statistics of the \emph{entire} sequence~$\tilde D(n)$.

\begin{theorem}\label{THM:ORDERSTAT} 
Let~$\tilde Y$ and~$\tilde D(n)$ be the infinite extensions of the random vectors~$Y$ 
and~$D(n)$ of Theorem~\ref{THM2} as just described.
Then~$\tilde Y\in l_p^+$ almost surely for any~$p>(\ell+1)/\ell$. Moreover, for~$\ell=1$ and any~$p\geq 4$,
or~$\ell\geq 2$ and any~$p\geq \ell+1$, the sequences~$\law(\tilde D(n))$ and~$\law(\tilde D(n)^\da)$
converges weakly to~$\law(\tilde Y)$ and~$\law(\tilde Y^\da)$ with respect to the~$l_p$-topology. 
\end{theorem}

\subsection{Idea of the proofs of Theorem~\ref{THM2} and Theorem~\ref{THM:ORDERSTAT}.}
The representation of the limit vector in Theorem~\ref{THM2} is integral to our approach. It says that in the limit for~$k>s$,
conditional on coordinate~$Z_{k+1}$, the previous coordinate~$Z_k$ is an independent beta variable multiplied by~$Z_{k+1}$.
On the other hand, for~$k>s$, conditional on the sum of the weights
of the first~$k+1$ vertices, say~$S_{k+1}(n)$, it is not too difficult to see that the sum of the weights of the first~$k$ vertices,~$S_k(n)$,
will be distributed as a classical P\'olya urn, run for a number of steps of the order~$S_{k+1}(n)$ (see Lemma~\ref{lem1} below).
 Furthermore, P\'olya urns limit to beta variables, so~$\law\bklr{S_{k}(n)}\approx \law\bklr{B_k S_{k+1}(n)}$, where~$B_k$ is an
 appropriate beta variable
independent of~$S_{k+1}(n)$ (see Lemma~\ref{lem2} below, an extension of \cite[Lemma~4.4]{Pekoz2013b}, which gives a new bound on the Wasserstein distance between P\'olya urns and beta distributions). The limits satisfy this approximate identity exactly; that is,
$Z_{k}=B_k Z_{k+1}$.  Once we take care of the error made in swapping out P\'olya urns for betas (done via a telescoping sum argument),
all that is left is to show that~$Z_r$ is close to~$S_r(n)$, which follows by results of \cite{Pekoz2013b}.
For Theorem~\ref{THM:ORDERSTAT}, we  establish
 weak convergence of the distribution of the sequence~$\tilde D(n)$
 to the distribution of~$\tilde Y$ with respect to~$l_p$ topology 
 by verifying a tightness criterion for probability measures on~$l_p$,
 and then applying Theorem~\ref{THM2} to show the convergence
of finite dimensional distributions.
The convergence of the distribution of
$ \tilde D^{\da}(n)$ follows from
the continuous mapping theorem applied to the ordering function.

\subsection{Related work}

To our knowledge all of the results above for~$\ell\geq 2$ are new.
The sequential-edge model is studied by \cite{Berger2014} where they give a beautiful
description of the limiting structure of the graph started from a random vertex and performing a depth first
search. That work is complementary to ours: as the number of vertices~$n\to\infty$,
the depth first search started from a uniformly chosen vertex only sees vertices of order~$n$,
and these do not appear in our limits.

For~$\ell=1$, 
the limiting marginal distribution of the degree of the~$i$th vertex 
 are described in different ways in  \cite{Janson2006} (after relating these variables to an appropriate triangular urn model) and \cite[Theorem 1.1, Proposition 2.3]{Pekoz2013}. 
 The latter's Models~1 and~2 are our Models N$_1$ with~$s=2, d_1=d_2=1$ and L$_1$ with~$s=1, d_1=2$, respectively,
and their scaling is a constant times~$\sqrt{n}$. For Model~N$_1$, they identify the distributional limit of~$D_i(n)$ for~$i\geq 2$ (note that Vertex 1 and 2 have the same distribution in this model)
as~$\klr{ \Gamma_1 B_{1/2, i-3/2}}^{1/2}$, where~$\Gamma_1\sim\Ga(1,1)$ is independent of~$B_{1/2, i-3/2}\sim\B(1/2, i-3/2)$.
In Model~L$_1$, the analogous limit can be written as~$\klr{ \Gamma_1 B_{1/2, i-1}}^{1/2}$ (interpreting~$B_{1/2,0}=1$).
Using our description of these limits given by Proposition~\ref{thm1}, we find that
for~$X_{1/2}\sim\Ga(1/2,1)$ independent of~$X_1,X_2,\ldots$, i.i.d.\ 
random variables with distribution~$\Ga(1,1)$,
\ba{
\sqrt{X_{1/2}+X_1+\cdots+X_{i-1}}-\sqrt{X_{1/2}+X_1+\cdots+X_{i-2}}&\ed \sqrt{ \Gamma_1 B_{1/2, i-3/2}}, \\
\sqrt{X_{1}+X_2+\cdots+X_i}-\sqrt{X_{1}+X_2+\cdots+X_{i-1}}&\ed \sqrt{ \Gamma_1 B_{1/2, i-1}}.
}
These intriguing distributional identities are not easily interpretable, but they can be directly verified by comparing Mellin transforms.
It would be of interest to obtain a representation of the joint distributions similar in appearance to the right hand side of these identities.

Still assuming~$\ell=1$,
there are some  existing results about properties and characterizations of limits of the joint 
degrees of fixed vertices and maximums of such, but only started from
certain seed graphs. Especially close to our work
in this special case is \cite{Mori2005}, who
uses martingale arguments to show that~$(2\tilde D(n))_{n\geq1}$, in Model~N$_1$ 
with~$s=2$ and~$d_1=d_2=1$, has almost sure limit denoted
$(\zeta_1,\zeta_2,\ldots, )$, which must have the same distribution 
as~$2\tilde Y$, though the description of the limiting~$\zeta_i$ previously given
is not so explicit: moment formula are given (and can be checked
to agree with those of~$2 \tilde Y$) and some other 
properties are derived. For example,   
\cite[Lemma 3.4 with~$\beta=0$]{Mori2005} shows that the variables
\be{
\tau_{j}:=\frac{\zeta_1+\cdots+\zeta_{j-1}}{\zeta_1+\cdots+\zeta_{j}}
}
are~$\B(2j-1,1)$ and that~$\tau_1,\dots,\tau_r, \zeta_1+\cdots+\zeta_r$ are independent. 
This coincides with our description of the limits in Theorem~\ref{THM2}.
See also \cite[Section~2]{James2015} for discussion of these representations in this and related models.
Again using martingales, \cite[Theorem~3.1]{Mori2005} 
shows~$(\max_{i\geq 1} D_i(n))_{n\geq 1}$
converges almost surely and in~$L_p$ for~$p\geq 1$ to~$\max_{i\geq1} \zeta_i$.
Thus we can  identify the distribution of this limit as that of~$2\max_{i\geq1} Y_i$.

Our work generalizes and extends these existing results for the~$\ell=1$
case in several directions. The rates of
convergence in 
Theorem~\ref{THM2} and Corollary~\ref{cor1} are new, as are
the descriptions of the limiting joint degrees, even for simple
seed graphs where existing results are available. 
At the process level, our~$l_p$ convergence complements 
the almost sure convergence of the sequence and maximum
given by \cite{Mori2005} for the basic seed graph, as well 
as covering much more, allowing different seed graphs and
giving convergence of the order statistics.

\paragraph{A different multi-edge preferential attachment graph.}
An alternative way to define a preferential attachment model where each new node attaches to~$\ell>1$ nodes
is given by \cite{Bollobas2001}.
The model begins by generating a random graph according to Model~L$_1$ or~N$_1$ with~$n\ell$ nodes, denoted~$G(n\ell)$, and
then for each of~$i=1,\ldots, n$, collapsing nodes~$(i-1)\ell+1, \ldots, i\ell$ into one node keeping all of the edges (so there may be loops in both models).
But with this definition the degree of the~$i$th node 
is just the sum of the degrees of nodes~$(i-1)\ell+1, \ldots, i\ell$
in~$G(n\ell)$ and so the finite dimensional limits can be read from Theorems~\ref{THM2}.
 Moreover, since a linear transformation
of a convex set is convex, the analogous error rates of the theorem in this more general setting also hold.
An import remark is that this multi-edge preferential attachment graph is fundamentally different than that introduced above:
in this model the degree of a fixed vertex grows like~$\sqrt{n}$ for any~$\ell$ where as in Models~N$_\ell$ and~L$_\ell$ the degree grows like~$n^{\ell/(\ell+1)}$.
Also note that using the representation of Proposition~\ref{thm1},
summing the limiting distributions of the degrees of adjacent vertices has a particularly simple form due to telescoping.
For example for Model~L$_1$ started from a single loop (so~$s=1$ and~$d_1=2$), the joint distributional limits of the scaled degrees of nodes~$i=1,\ldots,r$ in~$G(n)^{(\ell)}$
are given by
\be{
\sqrt{X_1+\cdots+X_{\ell i}}-\sqrt{X_1+\cdots+X_{\ell (i-1)}}
}
where~$X_1,X_2,\ldots$ are i.i.d.\ distributed~$\Ga(1,1)$.

Similarly, at the process level, the lumping operation maps~$l_p^+$ to~$l_p^+$
since
by H\"older's inequality, for numbers~$x_1,\ldots, x_\ell$, 
\be{
\abs{x_1+\cdots+x_\ell}^p\leq \ell^{p-1} (\abs{x_1}^p+\cdots+\abs{x_\ell}^p).
}
The lumping is also Lipschitz continuous since for~$x,y$ in~$l_p^+$, 
\be{
\sum_{i\geq 1} \bbbabs{\sum_{j=1}^\ell x_{ (i-1) \ell+j}-\sum_{j=1}^\ell y_{(i-1)\ell +j}}^p
\leq \sum_{i\geq 1} \sum_{j=1}^\ell \babs{x_{ (i-1) \ell+j}- y_{(i-1)\ell +j}}^p.
}
So by the continuous mapping theorem, we can easily read the limits
of the degrees for the lumped graph from those of the original graph
given in Theorem~\ref{THM:ORDERSTAT}. 
We don't provide a formal statement of our results for this case because it's a
straightforward derivative of the~$\ell=1$ case.

\paragraph{Connection to Brownian CRT}
Consider Model~L$_1$ started from a loop (so~$s=1$ and~$d_1=2$). If we write~$S_i(n)=\sum_{j=1}^i D_j(n)$, then Proposition~\ref{thm1}
and the results of \cite{Mori2005} discussed above
imply that for~$r\geq1$, the scaled sums of degree counts
\be{
(S_1(n),\ldots, S_r(n))\stackrel{a.s.}{\longrightarrow} (\sqrt{X_1},\sqrt{X_1+X_2}, \ldots, \sqrt{X_1+\cdots+X_r}),
}
where the~$X_i$ are i.i.d.\ and have distribution~$\Ga(1,1)$.
These are the points of an inhomogeneous Poisson point process with intensity~$ 2t\,dt$, which
also arises in Aldous's CRT construction (see \cite{Aldous1991, Aldous1993}), 
and is described around \cite[Theorem~7.9]{Pitman2006}. 
The explicit connection is that if we consider~$G(n)$ plus the not yet attached half edge of vertex~$n+1$,
then there are~$2n+1$ ``degrees" which can be bijectively mapped to a binary tree with~$n+1$ leaves.
The bijection is defined through R\'emy's algorithm for generating uniformly chosen binary plane trees (see the discussion \cite[Remark 2.6]{Pekoz2013b}). The algorithm begins with a binary tree with two leaves and a root, corresponding to the two starting degrees of the loop and the half edge of the second vertex in Model~L$_1$. In R\'emy's algorithm, leaves are added to the tree by selecting a (possibly internal) vertex uniformly at random and and inserting a cherry at this vertex (that is, insert a
graph with three vertices and two edges with the ``elbow" oriented towards the root of the binary tree),
so two vertices (one of which is a leaf) are added at each step and these correspond to the two degrees of each edge added in the preferential attachment model. The number of vertices in the spanning tree of the first~$k$ leaves added in R\'emy's algorithm
is exactly the sum of the degrees of the first~$k$ vertices in Model~L$_1$ started from a loop.
If the leaves are labeled in the order they appear (the initial two leaves labeled~$1$ and~$2$),
then this leaf labeling is uniform which implies that if we first choose a uniform random binary plane tree with~$n \geq 2~$ leaves and then~$k\leq n$ leaves uniformly at random and fix a labeling~$1,\dots, k$, then for~$T_j(n)$ defined to be the number of vertices in the spanning tree containing  the root and the leaves labeled~$1,\ldots, j$ we have
\be{
(S_1(n),\ldots, S_k(n))\ed \frac{1}{2 n^{1/2}} (T_1(n),\ldots, T_k(n)).
}
Theorem~\ref{THM2} provides a rate of convergence of
this random vector to its limit. 

We can now clearly see the connection to the CRT
since according to \cite{Aldous1993}, 
uniform random binary plane trees converge to Brownian CRT, and the number of vertices in the spanning
tree of~$k$ randomly chosen leaves
in a uniform binary tree of~$n$ leaves  converges to the length of the tree induced by Brownian excursion sampled at
$k$ uniform times as per \cite[Theorem~7.9]{Pitman2006}, \cite{Pitman1999c}, and it is known that these
trees are formed by combining branches with lengths given by a Poisson point process on  the positive line
with intensity proportional to~$t\, dt$.

\paragraph{A statistical application.}
Is it possible with probability greater than~$1/2$ to tell the difference between two preferential attachment graphs started from non-isomorphic seeds
and run for a long time? This question is posed in the~$\ell=1$ case by \cite{Bubeck2014},
where the answer was determined to be yes as long as the degree sequences of the seeds are different.
The crucial step is to separate the two graphs based on the maximum degree which relies on a careful understanding of
the maximum degree along the lines of Corollary~\ref{cor1}. 
In fact by exploiting the connection to the Brownian CRT just mentioned (along
with other results), the answer to the question can be strengthened to yes as long as the two seed graphs are non-isomorphic, see
\cite{Curien2014}.

\paragraph{Organization of the paper.}
To show Theorems~\ref{THM2} and~\ref{THM:ORDERSTAT}, we relate the weight distributions
of both Models N$_\ell$ and L$_\ell$  to a single infinite color urn model that generalizes the single color models considered by \cite{Janson2006}, \cite{Pekoz2013}, \cite{Pekoz2013b}; urn models frequently appear when studying preferential attachment,
see for example \cite{Antunovic2013}, \cite{Berger2014}, \cite{Pekoz2013}, \cite{Pekoz2013a}, \cite{Ross2013},
and \cite{Pemantle2007}.
In the next section we define the relevant infinite color urn model
and make explicit the equivalence to the preferential attachment models under study.
In Section~\ref{sec:THM2proof} 
we state and prove a general approximation result from which Theorem~\ref{THM2} follows.
Section~\ref{sec:processconv} has the proof of Theorem~\ref{THM:ORDERSTAT}.

\subsection{An infinite color urn model}\label{sec1}

Consider the following infinite-color urn model. 
At step~$0$ there are ~$s\geq 1$ distinct colors present in the urn, and we assume that these colors are labelled from~$1$ to~$s$.
Fix an integer~$\ell\geq1$. At the~$n$th step, a ball is picked at random from the urn and returned along with an additional ball of the same color. Additionally, if~$n$ is a multiple of~$\ell$, a ball of (the new) color~$s+n/\ell$ is added after the~$n$th draw.

We are interested in the cumulative color counts; that is, for each~$k\geq 1$, let~$M_k(n)$ be the number of balls of colors~$1$ to~$k$ in the urn after the~$n$th draw  (and possible immigration) has been completed.
 Let~$m_k$ be the number of balls of colours~$1$ to~$k$ at time~$0$, so that~$M_k(0) = m_k$. 
 In order to avoid degeneracies, we will assume that, at time zero, at least one ball of each color from~$1$ to~$s$ is present; that is,~$m_1\geq1$ and~$m_{k+1}-m_{k}\geq 1$ for~$1\leq k < s$.

The following result makes explicit the connection between the urn model just described and 
the preferential attachment models under study. It follows from straightforward considerations.

\begin{lemma}\label{lemlink}
 Fix a seed graph~$G(0)$ with~$s$ vertices, and let~$d_1,\dots,d_s$ be the initial weight sequence. Let
$\ell\geq 1$ and consider either the situation of
\begin{enumerate}[$(i)$]
\item Model~N$_{\ell}$ for the graph sequence, and an infinite color urn model with initially~$d_k$~balls of color~$k$, where~$1\leq k\leq s$;
or
\item Model~L$_{\ell}$ for the graph sequence, and an infinite color urn model with initially~$d_k$~balls of color~$k$, where~$1\leq k\leq s$, and one ball of color~$s+1$.
\end{enumerate}
Then, for any~$r$ and~$n > r$,
\be{
	\bklr{d_1(n),d_1(n)+d_2(n),\dots,d_1(n)+\dots+d_r(n)} \eqlaw \bklr{M_1(\ell n),\dots,M_r(\ell n)}.
}
\end{lemma}

\section{Proof of Theorem~\ref{THM2}}\label{sec:THM2proof}

Theorem~\ref{THM2} easily follows from the next result (proved immediately after its statement),
the fact that linear transformations of convex sets are convex, and Lemma~\ref{lemlink}.

\begin{proposition}\label{PROP3} Fix~$r > s$. Let~$B_1,\dots,B_{r-1}$ and~$Z_r$ be independent random variables such that
\be{
	B_k\sim  \begin{cases}
	\B(m_{k},m_{k+1}-m_k) & \text{if~$1\leq k < s$,} \\
	\B(m_s+(\ell+1)(k-s)+\ell,1) & \text{if~$s\leq k < r$,}
	\end{cases}
}
and~$Z_r \sim \GG(m_s+(\ell+1)(r-s)+\ell,\ell+1)$. Define
\be{
	Z_k = B_{k}\cdots B_{r-1} Z_r, \quad 1\leq k < r, \qquad Z=(Z_1,\dots,Z_r),
}
and let
\be{
	W(n)=(W_1(n),\ldots,W_r(n))=\frac{\ell^{\ell/(\ell+1)}}{(\ell+1)n^{\ell/(\ell+1)}}\bklr{M_1(n),\dots,M_r(n)}.
}
Then there is a positive constant~$C=C(r,\ell,m_s)$, independent of~$n$, such that
\ben{\label{1}
	\sup_{K}\babs{\IP\kle{W(n)\in K} - \IP\kle{Z\in K}} \leq \frac{C}{n^{\ell/(\ell+1)}}
}	
for all~$n$, where the supremum ranges over all convex sets~$K\subset \IR^{r}$.

\end{proposition}

We need some intermediate lemmas to prove Proposition~\ref{PROP3}. 
Denote by~$\bPclas{b}{w}{m}$ the distribution of white balls in a classical P\'olya urn after~$m$ completed draws, starting with~$b$ black and~$w$ white balls. Denote by~$\bPimm{\ell}{b}{w}{m}$ the number of white balls after~$m$ completed steps in the following P\'olya urn with immigration, starting with~$b$ black and~$w$ white balls: at the~$n$th step, a ball is picked at random from the urn and returned along with an additional ball of the same color;
additionally, if~$n$ is a multiple of~$\ell$, then a black ball is added \emph{after} the~$n$th draw and return.

\begin{lemma} \label{lem1} Let~$\ellp =  k-s+1$.
 If~$k \geq s$, we have
\ben{\label{3}
	M_k (n) \sim \bbPimm{\ell}{1}{m_s+\ell \ellp+(k-s)}{n-\ell \ellp}.
}
Furthermore, conditionally on~$M_{k+1}(n)$, we have
\ben{\label{4}
	M_k(n)\sim \begin{cases}
	\bbPclas{m_{k+1}-m_k}{m_{k}}{M_{k+1}(n)-m_{k+1}} & \text{if~$1\leq k < s$,}  \\[3ex]
	\bbPclas{1}{m_s+\ell \ellp+(k-s)}{M_{k+1}(n)-m_s-(\ell+1)\ellp} & \text{if~$k\geq s$.}
	\end{cases}
}
\end{lemma}
\begin{proof}
To prove \eq{3} note that the number of balls having a color in the set~$\{1,\dots,k\}$ is deterministic up to the point where the first ball of color~$k+1$ appears in the urn; this is the case after~$\ell\ellp=\ell(k-s+1)$ completed draws. At that time we have~$M_{k}(\ell\ellp)+1=m_s+(\ell+1)\ellp$, so that~$M_{k}(\ell\ellp) = m_s+\ell\ellp+(k-s)$. After that, consider all  balls of colors~$\{1,\dots,k\}$ as `white' balls and  all balls of colors~$\{k+1,\dots,\}$ as `black' balls. The number of `white' balls for the remaining~$n-\ell\ellp$ steps now behaves exactly like~$\bPimm{\ell}{b}{w}{m}$ with~$b=1$,~$w=m_s+\ell\ellp+(k-s)$ and~$m=n-\ell \ellp$.

To prove the second line of \eq{4}, consider all balls of colors~$\{1,\dots,k\}$ as `white' balls and balls of color~$k+1$ as `black' balls. After time~$\ell\ellp$ the number of `white' balls now behaves exactly like~$\bPclas{b}{w}{m}$ with~$b=1$,~$w=M_{k}(\ell\ellp)$, and~$m$ being the number of times a ball among colors~$\{1,\dots,k+1\}$ was picked after time~$\ellp$, which is just~$M_{k+1}(n)-m_s-(\ell+1)\ellp$. 

The argument to prove the first line of \eq{4} is similar and therefore omitted.
\end{proof}

We will need the following coupling of P\'olya urns and beta variables
that is a generalization (from the~$b=1$ case) of \cite[Lemma~4.4]{Pekoz2013b}; for related distributional approximation
results, see \cite{Goldstein2013}.
\begin{lemma}
\label{lem2}
Fix positive integers~$b$,~$w$ and~$n$. There is a coupling~$(X, Y)$ with~$X\sim \bPclas{b}{w}{n}$ and~$Y\sim \B(w,b)$, such that almost surely,
\ben{\label{6}
	\abs{X-nY}< \frac{b(4w+b+1)}{2}.
}
\end{lemma}
\begin{proof} We use induction over~$b$, and start with the base case~$b=1$. Let~$V_0,\dots,V_{w-1}$ be independent and uniformly distributed on the interval~$[0,1]$. By a well known representation of the distribution~$\B(w,1)$, we can choose
\be{
	Y:=\max(V_{0},\ldots , V_{w-1}).
}
To construct~$X$, first note that for~$\bPclas{1}{w}{m}\{A\}$ denoting
the probability the relevant P\'olya urn law puts on the set~$A\subset \IZ$, we have
\be{
	\bPclas{1}{w}{m}\{w,\dots,t\}=\prod_{k=0}^{w-1} \frac{t-k}{m+w-k}
}
for all~$m\geq 0$ and for~$w\leq t \leq w+m$ (see e.g.\ \cite[Eq.~(2.4), p.~121]{Feller1968}). For each~$m\geq 0$, let
\be{
	N(m) :=\max_{0\leq k\leq w-1} \bklr{k+\ceil{ (m+w-k)V_k}}.
}
It is easy to see that the cumulative distribution function of~$N(m)$ is that of~$\bPclas{1}{w}{m}$ for each~$m$, and that
\ben{\label{7}
	\abs{N(m)-mY} \leq w+1 \qquad\text{for all~$m\geq 0$.}
}
Letting~$X:=N(n)$, \eq{6} follows for the case~$b=1$. As a side remark, note that, although~$N(m)\sim \bPclas{1}{w}{m}$ for each~$m$, the joint distribution of~$(N(0),N(1),\dots)$ is not that of a P\'olya urn process!

To prove the inductive step, assume we have constructed~$N_{b-1}(0),N_{b-1}(1),\dots$ and~$Y_{b-1}$ such that~$N_{b-1}(m)\sim\bPclas{b-1}{w}{m}$ for all~$m\geq 0$, such that~$Y_{b-1}\sim \B(w,b-1)$, and such that
\ben{\label{8}
	\abs{N_{b-1}(m)-mY_{b-1}} \leq (b-1)(4w+b)/2
}
for all~$m\geq 0$. Now, let~$Y'_b\sim \B(w,1)$ be independent of all else and let~$N'(0),N'(1),\dots$ be defined and coupled to~$Y'$ as in the base case, that is~$N'(m)\sim\bPclas{1}{w}{m}$ and~$\abs{N'(m)-m Y'}\leq w+1$. Define
\be{
	N_{b}(m):= N'\bklr{N_{b-1}(m)-(w+b-1)}.
}	
It is not difficult to see that~$N_{b}(m)\sim\bPclas{b}{w}{m}$. Also, it is not difficult to see that~$Y_b := Y_{b-1}Y'\sim\B(w,b)$. Noting from \eq{7} that  for any~$y>0$
\be{
	\abs{N'(m)-y Y'}\leq \abs{N'(m)-m Y'} + \abs{m-y}Y'\leq  (w+1) + \abs{m-y},
}
we have
\bes{
	\abs{N_b(m) - mY_b} & = \babs{N'\klr{N_{b-1}(m)-(w+b-1)} - mY_{b-1}Y'}  \\
	& \leq (w+1) + \babs{N_{b-1}(m)- (w+b-1)- mY_{b-1})} \\
	& \leq (w+1) + (w+b-1) + (b-1)(4w+b)/2 \\
	& = \frac{b(4w+b+1)}{2}.
}
This concludes the inductive step, where \eq{6} is just the case~$m=n$.
\end{proof}

The last ingredient we need to prove Proposition~\ref{PROP3}
is some moment information.

\begin{lemma} \label{lem3}
For any~$k > s$, and~$q=1,2,\ldots,$ we have
\ben{
	\IE M_k(n)^q \asymp n^{q\ell/(\ell+1)}, \label{9}
}
and for ~$w=m_s+(\ell+1)(k-s)+\ell$ and~$t=n-\ell(k-s+1)$,
\besn{\label{10}
		&\IE\klg{ M_k(n)(M_k(n)+1)\cdots (M_k(n)+\ell)}   \\
		&\qquad\quad=w\frac{(w+1+\floor{\frac{t-1}{\ell}}+t)\cdots(w+1+\floor{\frac{t-1}{\ell}}+t+\ell)}
{w+1+(\ell+1)\floor{\frac{t-1}{\ell}}+\ell} \\
		&\qquad \quad= \bklr{m_s+(\ell+1)(k-s)+\ell}n^\ell \left(\frac{\ell+1}{\ell}\right)^{\ell}(1+\bigo(n^{-1})).
}
Furthermore,
\ben{\label{11}
	\limsup_{n\toinf}\IE\bbbklg{\frac{n^{\ell/(\ell+1)}}{ M_k(n)}} <\infty.
}
\end{lemma}
\begin{proof} The asymptotic \eq{9} follows from \cite[Lemma~4.1]{Pekoz2013b}. From that lemma, we also have
for~$Y\sim \bPimm{\ell}{1}{w}{t}$,
\be{
	\IE \klg{Y(Y+1)\cdots (Y+\ell)}
	=\prod_{j=0}^\ell (w+j) \prod_{i=0}^{t-1}\left(1+\frac{\ell+1}{w+1+i+\lfloor i/\ell\rfloor}\right).
}
Setting~$T = \floor{\frac{t-1}{\ell}}$, using careful bookkeeping and a telescoping argument, we obtain
\ba{
&\IE \klg{Y(Y+1)\cdots (Y+\ell)}\\
&\qquad=\prod_{j=0}^\ell (w+j) \prod_{i=0}^{t-1}\left(1+\frac{\ell+1}{w+1+i+\lfloor i/\ell\rfloor}\right)\\
&\qquad=\prod_{j=0}^\ell (w+j) \prod_{k=0}^{T -1}\prod_{i=0}^{\ell-1} \left(1+\frac{\ell+1}{w+1+i+k(\ell+1)}\right)
	 \prod_{m=\ell T}^{t-1}\left(1+\frac{\ell+1}{w+1+m+T}\right)  \\
&\qquad=w \prod_{i=0}^{\ell-1}(w+1+i +(\ell+1)T) 
	 \prod_{m=\ell T}^{t-1}\left(1+\frac{\ell+1}{w+1+m+T}\right)  \\
&\qquad=w \prod_{i=0}^{\ell-1}(w+1+i +(\ell+1)T) 
	 \prod_{m=0}^{t-1-\ell T}\left(1+\frac{\ell+1}{w+1+m+(\ell+1)T}\right)  \\
&\qquad=w \prod_{i=t-\ell T}^{\ell-1}(w+1+i +(\ell+1)T) 
	 \prod_{m=0}^{t-1-\ell T}\left(w+1+m+(\ell+1)(T+1)\right)  \\
&\qquad=w \smashoperator{\prod_{i=0}^{\ell+\ell T-t-1}}\enskip(w+1+i +T+t  ) 
	 \quad\smashoperator{\prod_{m=\ell+\ell T-t+1}^{\ell}}\enskip\left(w+1+m+T+t\right)  \\
&\qquad=w\frac{(w+1+T+t)\cdots(w+1+T+t+\ell)}
{w+1+(\ell+1)T+\ell}.
}
Setting~$w=m_s+(\ell+1)(k-s)+\ell$ and~$t=n-\ell(k-s+1)$ as per~\eq{3} yields~\eq{10}.

In order to prove \eq{11}, let~$X\sim \bPimm{\ell}{1}{w+1}{t}$ and~$Y\sim \bPimm{\ell}{1}{w}{t}$,
where~$w=m_s+(\ell+1)(k-s)+\ell-1$. From~\eq{3} and \cite[Lemma~4.2]{Pekoz2013b} we have that
\be{
	\IE f(X) = \frac{\IE\klg{Yf(Y+1)}}{\IE Y}
}
for any bounded function~$f$; in particular for the function~$f(x)=1/x$,  bounded when~$x\geq 1$, we have
\be{
	\IE X^{-1} = \IE\bbbklg{\frac{Y}{(Y+1)\IE Y}} \leq \frac{1}{\IE Y}.
}
By \eq{9},~$\IE Y\asymp n^{\ell/(\ell+1)}$, from which \eq{11} easily follows.
\end{proof}

Note that the Landau-$\bigo$ notation of the lemma contains a constant that depends on~$\ell,k,$ and both of~~$m_s,s$. But here and below we ignore the dependence of constants on~$s$ when also depending on~$m_s$ through the inequalities~$1\leq s\leq m_s$. 

\begin{proof}[Proof of Proposition~\ref{PROP3}] To ease notation, we fix~$n$ and drop it in our variable notation and write, for example,
$W_k$ instead of~$W_k(n)$. For~$k < l$, let
\be{
	K_{k,l} = B_k\cdots B_{l}, \qquad
	V_{k,l} = K_{k,l-1}W_l, \qquad V_{l,l} = W_l.
}
We may assume that~$B_1,\dots,B_{r-1}$  are all independent of~$W=W(n)$. Fix a convex subset~$K\subset \IR^r$, assuming without loss of generality that~$K$ is closed, and let~$h$ be the indicator function of~$K$. We proceed in two major steps by writing
\bes{
	&\IP\kle{W\in K} - \IP\kle{Z\in K}  \\
	& \quad = \IE h(W) - \IE h(Z) \\
	&\quad = \IE \bklg{h(W) - h\klr{V_{1,r},\dots,V_{r,r}}} + \IE \bklg{h\klr{V_{1,r},\dots,V_{r,r}} - h(Z)}  \\
	&\quad =: \IE R_1 + \IE R_2.
}
In order to bound~$R_1$, first write it as the telescoping sum
\bes{
	R_1 & = h(V_{1,1},\dots,V_{r,r})-h\bklr{V_{1,r},V_{2,r},\dots,V_{r,r}}\\
	&=
	\sum_{k=1}^{r-1} \bkle{
	 h\bklr{V_{1,k},\dots,V_{k-1,k},V_{k,k},V_{k+1,k+1},\dots,V_{r,r}} \\
	&\qquad\qquad -
	h\bklr{V_{1,k+1},\dots,V_{k-1,k+1},V_{k,k+1},V_{k+1,k+1},\dots,V_{r,r}}
	}  \\
	& =: \sum_{k=1}^{r-1} R_{1,k}.
}
Fix~$k$, and define~$b_k$ and~$w_k$ to be the parameters of the P\'olya urn in Lemma~\ref{lem1}, that is,
\be{
	b_k = \begin{cases}
	m_{k+1}-m_k & \text{if~$1\leq k < s$,} \\
	1 & \text{if~$k\geq s$,}
	\end{cases}
	\quad
	w_k = \begin{cases}
	m_k & \text{if~$1\leq k < s$,} \\
	m_s + \ell (k-s+1) +(k-s) & \text{if~$k\geq s$.}
	\end{cases}
}
Conditioning on~$M_{k+1}$, we can use Lemmas~\ref{lem1} and~\ref{lem2}, to conclude that there is a coupling~$(X_k,Y_k)$ with
\be{
	X_k  \sim \bbPclas{b_k}{w_k}{M_{k+1}-c_k},
	\qquad
	Y_k  \sim \B(w_k,b_k),
}
such that almost surely,
\be{
	\abs{X_k - (M_{k+1}-c_k)Y_k} \leq \frac{b_k(4w_k+b_k+1)}{2},	
}
where
\be{
	c_k =\begin{cases}
	m_{k+1} & \text{if~$1\leq k < s$,}\\
	m_s+(\ell+1)(k-s+1) & \text{if~$k\geq s$.}
	\end{cases}
}
Hence, there is a constant~$C(k,\ell,m_s)$ and a random variable~$E_k$ with~$\abs{E_k}\leq C(k,\ell,m_s)$ almost surely, such that
\be{
	\frac{X_k}{M_{k+1}} = Y_k + \frac{E_k}{M_{k+1}}.
}
Define
\ba{
	V'_{j,k} & :=\frac{\ell^{\ell/(\ell+1)}}{(\ell+1)n^{\ell/(\ell+1)}} K_{j,k-1} X_k = K_{j,k-1}\frac{X_k}{M_{k+1}} W_{k+1}  \quad\text{for~$1\leq j < k$,}\\
	V''_{j,k+1} & :=\frac{\ell^{\ell/(\ell+1)}}{(\ell+1)n^{\ell/(\ell+1)}}K_{j,k-1} Y_k M_{k+1} = K_{j,k-1} Y_k W_{k+1} \quad\text{for~$1\leq j < k + 1$,}
}
and
\be{
	V'_{j,j}  := V_{j,j}\quad\text{for~$j \geq  k$,}\qquad
  V''_{j,j}  := V_{j,j}\quad\text{for~$ j \geq  k + 1$.}
}
Note that by Lemma~\ref{lem1},~$\law(V'_{j,k}) = \law(V_{j,k})$ and~$\law(V''_{j,k+1}) = \law(V_{j,k+1})$, hence we have
\bes{
	R_{1,k} & \eqlaw h\bklr{V'_{1,k},\dots,V'_{k-1,k},V'_{k,k},V'_{k+1,k+1},\dots,V'_{r,r}} \\
	&\quad\qquad -
	 h\bklr{V''_{1,k+1},\dots,V''_{k-1,k+1},V''_{k,k+1},V''_{k+1,k+1},\dots,V''_{r,r}}\\
	 & = g(Y_k + E_k/M_{k+1}) - g(Y_k),
}
where we define (in notation anticipating future conditioning)
\be{
	g(x) := h(K_{1,k-1}xW_{k+1},\dots,K_{k-1,k-1}xW_{k+1},W_{k+1},\dots,W_{r}).
}
Let~$\cF_k$ be the~$\sigma$-algebra generated by~$B_1,\dots,B_{k-1}$,~$E_k$, and~$W$ (hence also~$(M_j)_{j=1}^r$).
Since~$h$ is the indicator function of a convex set, and since linear transformations of convex sets result again in convex sets, conditional on~$\cF_k$, we have that~$g$ is of the form~$g(x) = \I[a\leq x\leq b]$ for~$-\infty\leq a \leq b\leq \infty$ which are random but~$\cF_k$-measurable. Hence,
\bes{
	\IE \klg{\abs{R_{1,k}}|\cF_k} & \leq \IP\bbbkle{a-\frac{\abs{E_k}}{M_{k+1}}\leq Y_k \leq a+ \frac{\abs{E_k}}{M_{k+1}}\bbbmid\cF_k}\\
	&\quad\quad+ \IP\bbbkle{b-\frac{\abs{E_k}}{M_{k+1}}\leq Y_k \leq b+ \frac{\abs{E_k}}{M_{k+1}}\bbbmid\cF_k} \\
	& \leq 2	\sup_{x\in \IR}\IP\bbbkle{x-\frac{C}{M_{k+1}}\leq Y_k \leq x+ \frac{C}{M_{k+1}}\bbbmid\cF_k}
	 \leq \frac{C'}{M_{k+1}}
}
for some constant~$C'=C'(k,\ell,m_s)$, and where the last inequality follows from basic properties of the beta distribution. Hence, by~\eq{11} of Lemma~\ref{lem3},
\ben{\label{12}
	\IE \abs{R_{1,k}} \leq \IE\bbbklg{\frac{C'}{M_{k+1}}} \leq \frac{C''}{n^{\ell/(\ell+1)}}
}
for some constant~$C''=C''(k,\ell,m_s)$.

In order to bound~$R_2$, write
\bes{
	R_2 & = h\bklr{(K_{1,r-1},\dots,K_{r-1,r-1},1)W_r} - h\bklr{(K_{1,r-1},\dots,K_{r-1,r-1},1)Z_r} \\
	& = g(W_r) - g(Z_r),
}
where now
\be{
	g(x) = h\bklr{(K_{1,r-1},\dots,K_{r-1,r-1},1)x}.
}
Given~$K_{1,r-1},\dots,K_{r-1,r-1}$, the function~$g$ is again of the same form as in the first part of the proof, so that
\ben{\label{13}
	\IE\abs{R_2}\leq 2\dk\bklr{\law(W_r),\law(Z_r)},
}
where~$\dk$ denotes the Kolmogorov distance, i.e., the supremum distance  between  distribution functions.
Define the scaling constants
\be{
\tilde{\mu}_n = \frac{(\ell+1)n^{\ell/(\ell+1)}}{\ell^{\ell/(\ell+1)}}, \hspace{5mm}  \mbox{and} \hspace{5mm}
\mu_n = \left(\frac{(\ell+1)\IE M_r(n)^{\ell+1}}{m_s+(\ell+1)(r-s)+\ell} \right)^{1/(\ell+1)}.
}
Since the Kolmogorov distance is scale invariant,
\besn{\label{14}
	& \dk\bklr{\law(W_r),\law(Z_r)} \\
	& \quad = \dk\bklr{\law(M_r(n)/\mu_n),\law(Z_r\, \tilde{\mu}_n/\mu_n)}\\
	& \quad \leq \dk\bklr{\law(M_r(n)/\mu_n),\law(Z_r)} +
	\dk\bklr{\law(Z_r),\law(Z_r\, \tilde{\mu}_n/\mu_n)} := R_{2,1}+R_{2,2}.
}
From \cite[Theorem~1.2]{Pekoz2013b} we have that~$R_{2,1}=\bigo(n^{-\ell/(\ell+1)})$ and, noticing that the density of~$Z_r$ is bounded, standard arguments give~$R_{2,2}=\bigo(\abs{1-\tilde{\mu}_n/\mu_n})$. To bound~$\abs{1-\tilde{\mu}_n/\mu_n}$, use~\eq{9} and~\eq{10} to find
\ba{
\IE M_r(n)^{\ell+1}&=\IE\{M_r(n)\cdots (M_r(n)+\ell)\}+\bigo(n^{\ell^2/(\ell+1)})\\
	&=(m_s+(\ell+1)(r-s)+\ell)n^\ell\left(\frac{\ell+1}{\ell}\right)^{\ell}\bklr{1+\bigo(n^{-1})}+\bigo\bklr{n^{\ell^2/(\ell+1)}}\\
	&=(m_s+(\ell+1)(r-s)+\ell)n^\ell \left(\frac{\ell+1}{\ell}\right)^{\ell}+\bigo\bklr{n^{\ell^2/(\ell+1)}}.
}
Using this last expression, we easily find
\ba{
\frac{{\mu}_n}{\tilde{\mu}_n}=\bbklr{1+\bigo\bklr{n^{-\ell/(\ell+1)}}}^{1/(\ell+1)}.
}
Now, using that for~$0<x<1$,~$(1+x)^{1/(\ell+1)}-1\leq x/(\ell+1)$, we find~${\mu}_n/\tilde{\mu}_n = 1 + \bigo\bklr{n^{-\ell/(\ell+1)}}$ and so~$R_{2,2}=\bigo\bklr{n^{-\ell/(\ell+1)}}$.
Now collecting the bounds \eq{12}, \eq{13} and \eq{14}, proves~\eq{1}.
\end{proof}

\section{Proof of Theorem~\ref{THM:ORDERSTAT}}\label{sec:processconv}

Since Theorem~\ref{THM2} establishes convergence of finite
dimensional distributions, to prove the claimed convergence of
$(\tilde D(n))_{n\geq 1}$,
we need to show tightness in~$l_p$ for the relevant~$p$.
In the next section we state and prove an~$l_p$ tightness criteria in terms of moment 
conditions
 and then in Section~\ref{sec:tightmombd} we 
bound the relevant moments appropriately.
In Section~\ref{sec:ordcont} we establish the
continuity of the ordering 
function on~$l_p$ and put everything together to prove Theorem~\ref{THM:ORDERSTAT}
in Section~\ref{sec:ORDERSTATproof}.

\subsection{Tightness in~$l_p$}

The following result is essentially \cite[Section~6, Theorem~16]{Suquet1999}, but we give a self-contained proof. 

\begin{lemma}\label{lem:tightcrit}
 Let~$X(n) = (X_1(n),X_2(n),\dots)$,~$n\geq 1$, be a sequence of~$l_p$-valued random elements, where~$1\leq p<\infty$. If 
\be{
	(i) \enskip \text{$\sup_{n\geq 1} \sum_{i\geq 1}\IE\abs{X_i(n)}^p < \infty$}\qquad\text{and}\qquad 
	(ii) \enskip \text{$\lim_{k\toinf} \sup_{n\geq 1} \sum_{i\geq k}\IE\abs{X_i(n)}^p= 0$,} 
}
then the sequence of probability measures~$\bklr{\IP[X(n)\in \cdot\,]}_{n\geq 1}$ is tight in the~$l_p$-topology.
\end{lemma}
\begin{proof}\def\IM{\mathbbm{M}}
By the Frechet-Kolmogorov theorem,~$B\subset l_p$ is relatively compact if it is bounded and
\be{
	 \lim_{k\toinf}\sup_{x\in B}\sum_{i\geq k} \abs{x_i}^p =0.
}
Thus, for any~$C>0$ and any strictly increasing sequence of numbers~$1\leq k_1< k_2<\cdots$, the set
\be{
	\mathcal{K}\klr{C,(k_m)_{m\geq 1}} := \bbbklg{x\in l_p\,:\, 
				\text{$\displaystyle\sum_{i\geq 1}\abs{x_i}^p\leq C$, and~$\displaystyle\sum_{i\geq k_m}\abs{x_i}^p\leq \frac{1}{m}$ for all~$m\geq 1$}}
}
is~$l_p$-compact. 

Now, fix~$\eps>0$. From~$(i)$ we conclude that there is~$C$ such that
\be{
	\IP\bbkle{\,\sum_{i\geq 1}\abs{X_i(n)}^p > C\,} \leq \frac{1}{C}\sum_{i\geq 1}\IE\abs{X_i(n)}^p\leq \frac{\eps}{2}\qquad\text{for all~$n\geq 1$.}
}
From~$(ii)$ we conclude that there is a sequence~$1\leq k_1<k_2<\cdots$ such that, for all~$m\geq 1$, 
\be{
	\IP\bbkle{\,\sum_{i\geq k_m}\abs{X_i(n)}^p > \frac{1}{m}\,} \leq m\sum_{i\geq k_m}\IE\abs{X_i(n)}^p\leq \frac{\eps}{2^{m+1}}\qquad\text{for all~$n\geq 1$.}
}
For this~$C$ and~$(k_m)_{m\geq 1}$, we conclude that, for all~$n\geq 1$,
\bes{
	\IP\bkle{X(n)\not\in \mathcal{K}\klr{C,(k_m)_{m\geq 1}}} & \leq \IP\bbkle{\,\sum_{i\geq 1}\abs{X_i(n)}^p > C\,} + \sum_{m\geq 1}\IP\bbkle{\,\sum_{i\geq k_m}\abs{X_i(n)}^p > \frac{1}{m}\,}\\
	&\leq \frac{\eps}{2}+\frac{\eps}{2}\sum_{m\geq 1}\frac{1}{2^m} = \eps.\qedhere
}
\end{proof}

\subsection{Moment bounds}\label{sec:tightmombd}
The following moment bounds are used to apply Lemma~\ref{lem:tightcrit}.
\begin{lemma}\label{lem:tightmombd}
Let~$\ell\geq1$, and let~$D_k(n)$ be as in Theorem~\ref{THM2}. There is a constant~$C=C(\ell, m_s,s)$ such that
\be{
\IE D_k(n)^{\ell+1}\leq C k^{-\ell}.
}
Moreover, in the special cases~$N_1$ and~$L_1$, 
\be{
\IE D_k(n)^4\leq C k^{-2}.
}
\end{lemma}
\begin{proof}
By Lemma~\ref{lemlink}, it's enough to prove the theorem for 
the appropriate equivalent urn model (possibly adjusting~$s$ and~$m_s$, each by at most one).
Recall the definition of~$M_k(n)$ as in Proposition~\ref{PROP3},
set~$U_1(n)=M_1(n)$, and for~$k=2,3,\ldots,$ let~$U_k(n)=M_k(n)-M_{k-1}(n)$,
where we set~$U_k(n)=0$ if~$k$ is a color that has not been added by time~$n$.
Note that~$\law(U_k(n))=\law(D_k(n))$.

First note that the statement of the theorem is trivial if~$U_k(n)=0$, that is, if color~$k$ has not yet appeared; in particular if~$k-s> n+1$.
Now, for~$k>s$, 
\be{
\law\bklr{U_k(n)|M_k(n)}=\bPclas{m_s+(k-s)(\ell+1)-1}{1}{M_k(n)-m_s-(k-s)(\ell+1)}.
}
Let the random variable~$B\sim\B[1, m_s+(k-s)(\ell+1)-1]$ be independent of~$M_k(n)$. Conditional on~$B$ and~$M_k(n)$, let 
$X(M_k(n), B)$ be binomial with parameters~$M_k(n)-m_s-(k-s)(\ell+1)$ and~$B$.
By the de Finetti representation of
the classical P\'olya urn, we have
\ben{\label{eq:polyarep}
\law\bklr{U_k(n)|M_k(n)}=\law\bklr{(1+ X(M_k(n), B))|M_k(n)}.
}
H\"older's inequality implies that for non-negative~$x,y$ 
\be{
(x+y)^p\leq 2^{p-1} (x^p + y^p),
}
and so starting from~\eq{eq:polyarep}, we have
\ben{\label{eq:ukmombd1}
\IE U_k(n)^p\leq 2^{p-1}\bklr{1+\IE  X(M_k(n),B)^p}.
}
Now note that, if~$\law(Y)=\Bi(N,q)$, 
then
for positive integer~$p$,
and denoting Stirling numbers of the
second kind by~${p \brace j}$ (and note these are non-negative),
\be{
\IE Y^p=\sum_{j=0}^p {p \brace j} \IE \klg{Y(Y-1)\cdots(Y-j+1)} \leq \sum_{j=0}^{p} {p\brace j} (Nq)^j.
}
So from~\eq{eq:ukmombd1}, condition on~$M_k(n)$ and~$B$ to find
\ben{\label{eq:ukmombd2}
\IE U_k(n)^p\leq 2^{p-1}\left(1+  \sum_{j=0}^{p} {p\brace j} \IE M_k(n)^j \IE B^j \right).
}
Standard formulas for beta moments imply
\ben{\label{eq:betamoms}
\IE B^j=\frac{\Gamma(j+1)\Gamma(m_s+(k-s)(\ell+1))}{\Gamma(m_s+(k-s)(\ell+1)+j)}\leq C k^{-j},
}
where~$C=C(\ell,m_s,s)$ is a constant.
By Jensen's (or H\"older's) inequality, for~$j\leq p$,
\ben{\label{eq:jensenmom}
\IE M_k(n)^j\leq \left(\IE M_k(n)^{p}\right)^{j/p}.
}
Set~$p=\ell+1$, and we use that~$M_k(n)^{\ell+1}\leq M_k(n)(M_k(n)+1)\cdots (M_k(n)+\ell)$.
From~\eq{10} of Lemma~\ref{lem3},
\bes{
&\IE\klg{M_k(n)(M_k(n)+1)\cdots (M_k(n)+\ell)}= (m_s+(\ell+1)(k-s)+\ell)\\
&\times \frac{(m_s+k-s+1+ \floor{\frac{n-\ell(k-s+1)-1}{\ell}}+n)\cdots(m_s+k-s+1+ \floor{\frac{n-\ell(k-s+1)-1}{\ell}}+n+\ell)}
{m_s+(\ell+1)(k-s+1)+(\ell+1) \floor{\frac{n-\ell(k-s+1)-1}{\ell}}+\ell}\\
&\leq (m_s+(\ell+1)k+\ell)\frac{(m_s+k-s+1+ \frac{n-\ell(k-s+1)-1}{\ell}+n+\ell)^{\ell+1}}{m_s+(\ell+1)(k-s+1)+(\ell+1) \frac{n-\ell(k-s+1)-1}{\ell}-1} \\
&\leq(m_s+(\ell+1)k+\ell) \frac{(m_s+ n\left(\frac{\ell+1}{\ell}\right)+\ell)^{\ell+1}}{m_s+n\left(\frac{\ell+1}{\ell}\right)-\frac{2\ell+1}{\ell}}\\
&\leq C' k n^\ell,
}
where~$C'=C'(\ell,m_s)$ is a constant.
Putting these moment estimates together with~\eq{eq:betamoms} and~\eq{eq:ukmombd2}, we find that
\ben{
\IE U_k(n)^{\ell+1}\leq C''\sum_{j=0}^{\ell+1} (k n^\ell)^{j/(\ell+1)} k^{-j}=C''\sum_{j=0}^{\ell+1} \left(\frac{n}{k}\right)^{ j\ell/(\ell+1)}
\leq C'''(n/k)^\ell, \label{eq:ukmombd3}
}
where~$C''$ and~$C'''$ are constants depending only on~$\ell, m_s, s$ and the last inequality is because we may assume~$n\geq k-s-1$.
Scaling~$U_k(n)$ by~$n^{-\ell/(\ell+1)}$, we have 
\be{
\IE \bbklg{\bklr{n^{-\ell/(\ell+1)} U_k(n)}^{\ell+1}}\leq C''' k^{-\ell},
}
Which proves the first assertion of the lemma.
For the second assertion, follow the proof up to~\eq{eq:jensenmom} and now set~$\ell=1$ and~$p=4$.
Again~$\IE M_k(n)^4$ can be bounded by the fourth rising factorial moment and now
\cite[Lemma~4.1]{Pekoz2013b} implies that for~$w=m_s+2(k-s)+1$,~$t=n-(k-s+1)$,
\bes{
&\IE\klg{ M_n(k) \cdots (M_n(k)+3)} \\
&\qquad= \prod_{j=0}^3 (w+j) \prod_{i=0}^{t-1} \frac{w+1+2i+4}{w+1+2i}\\
	&\qquad=w(w+2)(w+1+2t)(w+1+2t+2)\\
	&\qquad\leq (w+2)^2(w+2t+3) \\
	&\qquad=(m_s+2(k-s)+3)^2 (m_s+2(k-s)+1+2(n-(k-s)-1)+3)^2 \\
	&\qquad\leq c k^2 n^2,
}
where~$c=c(m_s, s)$ is a constant. Now applying the same argument as~\eq{eq:ukmombd3}, 
we have that for some constants ~$c', c''$ depending only on~$m_s, s$,
\be{
\IE U_k(n)^{4}\leq c' \sum_{j=0}^4 \IE\bklg{(n^2 k^2)^{j/4}} k^{-j} \leq c'' (n/k)^2.
}
Scaling~$U_k(n)$ in this last equation yields the second assertion.
\end{proof}

\subsection{Continuity of the ordering function}\label{sec:ordcont}

For~$x\in l_p^+$, define the ordering function~$x^\da\in l_p^{+}$
as follows.  Since~$\lim_{i\toinf} x_i = 0$, the sequence has a maximum. Let~$y_1$ be the value of that maximum, remove it from the sequence~$x$, and repeat, assigning consecutive values to~$y_2$,~$y_3$ 
and so forth. Then, set~$x^\da:= y$. Note that the ordering function is not just a permutation of the coordinates:
~$(1/2,0,1/4,0,1/8,0,\dots)^{\da} = (1/2,1/4,1/8,\dots)$. To see that~$ x^\da \in l_p^+$, note that
 ~$\norm{x^\da}_p = \norm{x}_p$, since every positive value appearing in~$x$ will appear 
  in~$x^\da$, and zero values in~$x$ do not contribute to the norm.

\begin{lemma}\label{lem:ordermap}
For any~$1\leq p< \infty$, the ordering function is continuous in~$l_p^+$.
\end{lemma}
\begin{remark}
As pointed out to us by a referee, the ordering function is in fact~$1$-Lipschitz in~$l_p$; this follows from \cite[Theorem~3.5]{Lieb2001} applied to step functions. We present the lemma and proof for the sake of completeness.
\end{remark}
\begin{proof}[Proof of Lemma~\ref{lem:ordermap}] 
Since~$l_p^+$ is a metric space, it is enough to show that the ordering function is sequentially continuous. 
Assume~$x(n)\to x$ in~$l_p^+$; that is,~$\norm{x(n)-x}_p\to0$
as~$n\to\infty$. For~$\alpha>0$
 (to be chosen later), define the set of indices 
\be{
	M_\alpha = M_\alpha(x)= \bklg{i\,:\,x_i>\alpha};
}
since~$\lim_{i\toinf} x_i = 0$, the set~$M_\alpha$ is finite for any positive~$\alpha$. Define the sequences~$x'$ and~$x''$ by~$x'_i = x_i\I[i\in M_\alpha]$ and~$x'' = x - x'$. Moreover, for each~$n$, define the two sequences~$x'(n)$ and~$x''(n)$ by~$x_i'(n) = x_i(n)\I[i\in M_\alpha]$ and~$x''(n)= x(n) - x'(n)$ (note that the sequence~$x(n)$ is decomposed with respect to~$M_\alpha$, which depends on~$x$ only). 

Now, fix~$\eps>0$, and choose~$\alpha$ such that
\be{
	\norm{x''}_p \leq \frac{\eps}{4},
}
and it is moreover possible to choose~$\alpha$ such that 
\ben{\label{431}
	\max_{i\not\in M_\alpha}x_i< \alpha < \min_{i\in M_\alpha} x_i.
}
Clearly, the first~$\abs{M_\alpha}$ elements of~$x^\da$ are just the ordered non-zero elements of~$x'$. Now, choose~$N$ large enough that
\ba{
	(i) & \enskip \text{$\norm{x - x(n)}_p\leq \eps/4$ for all~$n\geq N$,} \\
	(ii) & \enskip \text{$x_i(n) > \alpha$ for all~$i\in M_\alpha$ all~$n\geq N$,} \\ 	
	(iii) & \enskip \text{$x_i(n) < \alpha$ for all~$i\not\in M_\alpha$ and all~$n\geq N$, and} 	\\
	(iv) & \enskip \text{$\norm{(x')^\da - (x'(n))^\da}_p \leq \eps/4$ for all~$n\geq N$.}
}
Condition~$(i)$ can be achieved because~$x(n)\to x$; Conditions~$(ii)$ and~$(iii)$ can be achieved because~$x(n)\to x$ implies uniform coordinate-wise convergence and because of \eq{431}; Condition~$(iv)$ can be achieved again by uniform coordinate-wise convergence and the fact that the order-statistics of finitely many values can be approximated to any arbitrary precision. 
It follows that, for any~$n\geq N$, 
\bes{
	\norm{x^\da -  (x(n))^\da}_p 
	&\leq \norm{ (x')^\da - (x'(n))^\da}_p + \norm{ x''}_p + \norm{ x''(n)}_p 
	\leq \frac{\eps}{4} + \frac{\eps}{4} + \frac{\eps}{2} = \eps.\qedhere
}
\end{proof}

\subsection{Proof of Theorem~\ref{THM:ORDERSTAT}}\label{sec:ORDERSTATproof}

\begin{proof}
We first argue that almost surely~$Y\in l_p^+$ for~$p>\ell/(\ell+1)$.
First note that~$Y_i=Z_i-Z_{i-1}$ and, according to Theorem~\ref{THM2}, 
for~$i> s$,~$Z_i=\GG(a_s+(\ell+1)(i-s))$ and~$Z_{i-1}=B_{i-1} Z_i$,
where~$B_{i-1}\sim\B(a_s+(\ell+1)(i-s-1),1)$ and is independent of~$Z_i$.
Thus we can write~$Y_i= Z_i (1-B_{i-1})$
and note that~$(1-B_{i-1})\sim \B(1, a_s+(\ell+1)(i-s-1))$. 
From standard moment formula for gamma and beta distributions,
we have 
\bes{
\IE Y_i^p&=\frac{\Gamma(\frac{a_s}{\ell+1}+i-s+\frac{p}{\ell+1})}{\Gamma(\frac{a_s}{\ell+1}+i-s)}
			\frac{\Gamma(1+p)\Gamma( a_s+(\ell+1)(i-s-1)+1)}{\Gamma( a_s+(\ell+1)(i-s-1)+1+p)}\\
		&\leq C i^{-p\ell/(\ell+1)},		
}
where~$C=C(p,\ell,a_s,s)$ is a constant and the inequality follow from
$x^t \Gamma(x)/\Gamma(x+t)\to1$ as~$x\to\infty$. 
Clearly we can extend the inequality to all~$i\geq1$.
Therefore
\be{
\IE \sum_{i=1}^\infty Y_i^p = \sum_{i=1}^\infty \IE Y_i^p\leq C \sum_{i=1}^\infty i^{-p\ell/(\ell+1)},
}
and this last term is finite for~$p>(\ell+1)/\ell$. Since the expectation of the sum is finite,
the sum is almost surely finite.

The tightness of~$(\tilde D(n))_{n\geq 1}$ 
follows by applying the moment bounds of Lemma~\ref{lem:tightmombd}
in Lemma~\ref{lem:tightcrit} and then convergence to the claimed limit 
follows from the the convergence of finite dimensional distributions given by Theorem~\ref{THM2}.

Finally,~$(\tilde D(n)^\da)_{n\geq1}$ converges to~$\tilde Y^\da$ 
because of the convergence of~$(\tilde D(n))_{n\geq1}$,
the continuity of the order mapping given by Lemma~\ref{lem:ordermap}.
\end{proof}

\section*{Acknowledgments}
We thank Jim Pitman for the suggestion to study joint degree distributions for preferential attachment graphs and for pointers to the CRT literature, Sourav Chatterjee for the suggestion to look at the order statistics of the process,
Rongfeng Sun for helpful discussions,
and the anonymous referees for their valuable comments. 
EP thanks the Department of Mathematics and Statistics at the University of Melbourne for their hospitality during
a visit when much of this
work was completed, supported by a UofM ECR grant.
AR was supported by NUS Research Grant R-155-000-124-112.
NR was supported by ARC grant DP150101459.

\setlength{\bibsep}{0.5ex}
\def\bibfont{\small}


\end{document}